\documentclass[final]{dmtcs-episciences}
\usepackage[utf8]{inputenc}

\usepackage{caption}
\usepackage{subcaption}

\usepackage{amsmath}
\usepackage{amssymb}
\usepackage{amsthm}
\usepackage{mathtools}
\usepackage{array}

\newtheorem{thm}{Theorem}
\newtheorem{prop}{Proposition}
\newtheorem{conj}{Conjecture}

\newenvironment{tightcenter}{%
  \setlength\topsep{4pt}
  \setlength\parskip{10pt}
  \begin{center}
}{%
  \end{center}
}

\DeclareMathOperator{\arcsinh}{arcsinh}

\usepackage[round]{natbib}

\author{Nicholas R.~Beaton\affiliationmark{1}\thanks{Supported by Australian Research Council grant DE170100186.} \and Andrew R.~Conway\affiliationmark{1} \and Anthony J.~Guttmann\affiliationmark{1}}
\title{On consecutive pattern avoiding permutations of length 4, 5 and beyond}

\affiliation{School of Mathematics and Statistics, The University of Melbourne, Australia}

\received{2017-5-5}
\revised{2018-2-9}
\accepted{2018-2-14}

\begin{document}

\publicationdetails{19}{2018}{2}{8}{3311}

\maketitle

\begin{abstract}
We review and extend what is known about the generating functions for consecutive pattern avoiding permutations of length 4, 5 and beyond, and their asymptotic behaviour. There are respectively, seven length-4 and twenty-five length-5 consecutive-Wilf classes. D-finite differential equations are known for the reciprocal of the exponential generating functions for four of the length-4 and eight of the length-5 classes. We give the solutions of some of these ODEs. An unsolved functional equation is known for one more class of  length-4, length-5 and beyond. We give the solution of this functional equation, and use it to show that the solution is not D-finite. 

For three further length-5 c-Wilf classes we give recurrences for two and a differential-functional equation for a third. For a fourth class we find a new algebraic solution.

We give a polynomial-time algorithm to generate the coefficients of the generating functions which is faster than existing algorithms, and use this to (a) calculate the asymptotics for all classes of length 4 and length 5 to significantly greater precision than previously, and (b) use these extended series to search, unsuccessfully, for D-finite solutions for the unsolved classes, leading us to conjecture that the solutions are not D-finite. We have also searched, unsuccessfully, for differentially algebraic solutions. 
\end{abstract}

\section{Introduction}\label{sec:intro}
Let ${\mathcal S}_n$ denote the set of permutations of the first $n$ integers. Given permutations $\pi \in {\mathcal S}_n$ and $\sigma \in {\mathcal S}_k$, $k \le n,$
then $\sigma$ is said to be a {\em pattern} in $\pi$ if, for some subsequence of $\pi$ of length $k,$  all the elements of the subsequence occur in the same relative order as do the elements of $\sigma.$
For example, 312 occurs as a pattern in the permutation 15234 in three ways, as 523, 524 and 534. It occurs once as a {\em consecutive} pattern, as 523. In a consecutive pattern, the elements of $\sigma$ are contiguous. 

The pattern 312 occurs once in the permutation 3412, but does not occur as a consecutive pattern. The permutation 3412 is said to be a {\em consecutive pattern-avoiding permutation}. In this article we denote the set of pattern-avoiding permutations as $\text{Av}(312),$ and the set of consecutive pattern-avoiding permutations as $\text{c-Av}(312).$ To distinguish pattern-avoiding permutations from their consecutive counterpart, the former are sometimes referred to as {\em classical} pattern-avoiding permutations (PAPs). 

The (ordinary) generating function for classical PAPs (and in future we will drop the adjective {\em classical}) is written 
\[P_\sigma(z) = \sum_{n \ge 0} p_n(\sigma) \cdot z^n,\] 
where $p_n(\sigma)$ is the number of permutations of length $n$ that avoid the pattern $\sigma.$ If $P_{\sigma_1}=P_{\sigma_2},$ we say that $\sigma_1$ is {\em Wilf-equivalent} to $\sigma_2.$ It was independently conjectured by Stanley and Wilf in the 1980s and subsequently proved by \cite{RA99} and \cite{MT04} that $p_n(\sigma)^{1/n}$ is monotone increasing (Arratia) and bounded from above (Marcus and Tardos), so that 
\[\lim_{n \to \infty} p_n(\sigma)^{1/n}  = \lambda(\sigma)\]
exists. This limit is called the {\em Stanley-Wilf limit} or {\em growth rate} of the pattern $\sigma.$ Thus we see that classical pattern-avoidance severely attenuates the number of permutations of length $n,$ reducing them from $n!$ to something that grows only exponentially with $n.$ For classical patterns it is known that all six patterns of length 3 are Wilf-equivalent, that there are three different Wilf-classes of length 4, sixteen classes of length 5, ninety-one classes of length 6 and five hundred and ninety-five classes of length 7 (see \cite{K12}).

For {\em consecutive} PAPs, the reduction imposed by the requirement of pattern-avoidance is much less severe. The number of acceptable patterns in this case still grows factorially, so the appropriate generating function is the exponential generating function,
\[C_\sigma(z) = \sum_{n \ge 0} \frac{c_n(\sigma)}{n!} \cdot z^n,\]
where $c_n(\sigma)$ is the number of permutations of length $n$ that avoid the consecutive pattern $\sigma.$ If $C_{\sigma_1}=C_{\sigma_2},$ we say that $\sigma_1$ is {\em c-Wilf-equivalent} to $\sigma_2.$ Analogously to the Stanley-Wilf-Marcus-Tardos result for classical patterns, \cite{E06} proved that
\[\lim_{n \to \infty} \left ( \frac{c_n(\sigma)}{n!} \right )^{1/n}  = \rho_{\sigma}\]
exists, and further that $0.7839 < \rho_\sigma < 1,$ if $\sigma$ is a pattern of length greater than 3. 

For consecutive patterns it is known that there are two different c-Wilf classes of length 3, there are seven different c-Wilf-classes of length 4, there are twenty-five different c-Wilf-classes of length 5, and there are ninety-two different c-Wilf-classes of length 6.

It has also been proved by \cite{E12} that, among all consecutive patterns $\sigma$ of a given length $m,$ $\rho_\sigma \le \rho(1,2,3, \ldots, m),$ and that $\rho_\sigma \ge \rho(1,2,3, \ldots, m-2,m, m-1).$ A comprehensive review of consecutive PAPs is given by \cite{E15}.

Closely related to the value of $\rho_\sigma$ are questions of the form and nature of the generating function $C_\sigma(z)$. For example, can a closed-form solution to $C_\sigma(z)$ be found? If not, does $C_\sigma(z)$ satisfy a functional equation, or a differential one? Is $C_\sigma(z)$ rational, algebraic, differentiably finite, differentiably algebraic, or none of these?

In this paper we consider the enumeration of all c-Wilf-classes of length 4 and length 5. In earlier work \cite{EN03} found the solution for three of the seven length-4 classes. \cite{EN12} found the solution for a fourth class, and gave an (unsolved) functional equation for a fifth class. They conjecture that the solution to this fifth class is not D-finite. In Section~\ref{ssec:4classV} we prove this to be the case. Our proof relies on showing that the solution has an infinite number of singularities, while a D-finite function can only have a finite number of singularities. (These arise from the zeros of the  polynomial multiplying the highest derivative in the differential equation.) Another property of D-finite functions used in the proofs is that an exponential power series is D-finite if and only if the corresponding ordinary power series is D-finite.

We then find a natural generalisation to certain classes of consecutive patterns of length $m\geq 4$, with solutions and non-D-finiteness following in a similar way. Finally, we find recurrence relations for an even greater generalisation, though we do not know if these may lead to closed-form solutions. We also comment on the solution to the known ODEs.

In Section~\ref{sec:length5} we turn our attention to length-5 classes. The theorems of \cite{EN03, EN12} give D-finite ODEs for eight of the twenty-five classes. We find a new closed-form solution to one c-Wilf class, and conjecture a solution to a generalisation to length-$m$ classes. For two other classes we find new recurrences, one of which leads to functional-differential equation, though unfortunately we do not know how to solve it. For a fourth class we find, but cannot solve, a functional equation.

In Sections~\ref{sec:gen_series} and~\ref{sec:numerical_results} we consider the application of numerical methods. We first outline an efficient, polynomial time algorithm for enumerating the series coefficients, which can then be followed by use of the package \emph{gfun} in Maple to search for a D-finite solution. In this way we reproduce known solutions to eight of the twenty-five c-Wilf classes of length 5. We typically generate more than 100 coefficients for length-4 c-PAPs, and 70 coefficients for length-5 c-PAPs. We could readily generate much longer series (and ocassionally do so), but these are of adequate length for our purposes. The appropriate differential equations are (with one exception) generated by less than 20 series coefficients, and confirmed by the subsequent coefficients.

Finally, we use these enumerations to give estimates of the growth constants for all the length-4 and length-5 classes, accurate to 19 significant digits. The estimates of the growth constants derive from application of the method of differential approximants, see for example \cite{G89} for a description. While 19 significant digits have been quoted, greater precision can easily be achieved with the available data, but we assumed that 19 digits is sufficient for most practical purposes. We then use these estimates to calculate the amplitudes to 17 significant digits.

Similar numerical work has been previously presented by \cite{BNZ}, who give two algorithms for the generation of the series coefficients. These algorithms can, with some computational constraints, be used to prove that two series are in the same c-Wilf class. \cite{Zseries} gives the series coefficients to length 60 for all c-Wilf classes of length 4 c-PAPs, and to length 40 for all c-Wilf classes of length 5 c-PAPs. \cite{Zeqns} gives algebraic solutions or ODEs for the cluster generating functions for three of the seven length-4 Wilf classes, and for seven of the length-5 Wilf classes. For the other c-Wilf classes, functional equations with several catalytic variables are given. \cite{Zeqns} also give numerical estimates of the growth constants and amplitudes, accurate to 10 and 8 digits respectively. The algorithm we give is faster, but cannot be used to prove that two series are in the same c-Wilf class (though we have abundant evidence of this from the agreement of the first 70 or more coefficients).

\section{Consecutive PAPs of length 4}\label{sec:length4}

In the case of classical PAPs, it is well known  that the 24 possible PAPs of length four can be divided into three  Wilf classes. For consecutive length-4 PAPs, there are seven (see \cite{EN03}) equivalence classes, within which each c-PAP has the same asymptotic behaviour. These are:
\begin{align*}
\text{4.I: } &1234 \sim 4321 \\
\text{4.II: } &2413 \sim 3142 \\
\text{4.III: } &2143 \sim 3412  \\
\text{4.IV: } &1324 \sim 4231 \\
\text{4.V: } &1423 \sim 3241 \sim 4132 \sim 2314\\
\text{4.VI: } &1342 \sim 4213 \sim 4123 \sim 2431 \sim 3124 \sim 1432 \sim 2341 \sim 3214 \\
\text{4.VII: } &1243 \sim 3421 \sim 4312 \sim 2134
\end{align*}

\cite{EN03} proved two powerful theorems, allowing the authors to give the solution of three of these classes. The first theorem gives the solution for the case when $\sigma$ is the strictly increasing sequence $1,2,3, \ldots, m-1,m.$ Then $C_{m}(z)=1/w(z),$ where $w(z)$ is the solution of the D-finite differential equation
\[\sum_{i=0}^{m-1} w^{(i)}=0,\]
 with $w(0)=1$, $w'(0)=-1$, and $w^{(k)}=0$ for $2\le k \le m-2$. A probabilistic approach to the proof of this theorem is given by \cite{P13}.
 
The second theorem gives the solution when $\sigma$ is the sequence $1,2, \ldots, a-1,a,\tau,a+1,$ of length $m+2.$ Here $a$ and $m$ are positive integers with $a \le m,$ and $\tau$ is any permutation of $\{a+2,a+3, \ldots, m+2\}.$ Then $C_{m,a}(z)=1/w(z),$ where $w(z)$ is the solution of the D-finite differential equation
\[w^{(a+1)}+\frac{z^{m-a+1}}{(m-a+1)!}w' = 0.\]

The first theorem gives the solution for class 4.I, while the second theorem immediately solves classes 4.VI and 4.VII. The solutions are given by the following.
\begin{thm}[\cite{EN03}]\label{thm:4classes_167}
For class \textup{4.I}, the exponential generating function is given by $1/w(x)$, where
\[w^{'''}+w^{''}+w'+w = 0; \quad w(0)=1, \quad w'(0)=-1, \quad w^{''}(0) =0.\]
 For class \textup{4.VI}, the e.g.f. is given by $1/w(x)$, where
\[2w^{''}+x^2w' = 0; \quad w(0)=1, \quad w'(0)=-1.\]
 For class \textup{4.VII}, the e.g.f. is given by $1/w(x)$, where
\[w^{'''}+x w^{'} =0; \quad w'(0) = 0, \quad w(0)=1, \quad w'(0)=-1,  \quad w^{''}(0) =0.\]
\end{thm}

We point out that the solution to the above ODE for class 4.I is $$w(x)=\frac{1}{2}\left (e^{-x}-\sin(x)+\cos(x)\right ) = \sum_{n \ge 0} \frac{(4n+1-x)\cdot x^{4n}}{(4n+1)!}.$$
The solution of the second
ODE above can be expressed as integrals of Airy functions, while that of the third ODE can be expressed in terms of Heun T functions and their integrals.

Later work by \cite{EN12} give the solution for class 4.IV:
\begin{thm}[\cite{EN12}]\label{thm:4classes_4}
For class \textup{4.IV}, the exponential generating function is given by $1/w(x)$, where
\[xw^{(v)}+(x+3)w^{(iv)} + (6x+3)w^{'''}+(5x+6)w^{''}+(8x+3)w'+4xw = 0;\]
\[w(0)=1, \quad w'(0)=-1, \quad w^{''}(0) =0, \quad w^{'''}(0) =0, \quad w^{iv}(0) =1.\]
\end{thm}
This ODE can be solved in terms of integrals of BesselY and BesselJ functions.

Before we move on to new results, we will give an exceedingly brief overview of the use of Goulden-Jackson cluster method (see \cite{GJ79, GJ83} and also~\cite{EN12}). Let $\sigma$ be a length-$m$ consecutive permutation pattern. A $k$-cluster of length $n$ (with respect to $\sigma$) is any permutation $\pi$ of length $n\geq m$ such that
\begin{itemize}
\item $\pi$ contains precisely $k$ occurrences of $\sigma$,
\item every element of $\pi$ is part of at least one occurrence of $\sigma$, and
\item any two successive occurrences of $\sigma$ overlap by at least one element.
\end{itemize}
For example, 162534 is a 2-cluster of length 6 for the consecutive pattern 1423, while 1523746 is a 2-cluster of length 7.

Let $s_{n,k}$ be the number of $k$-clusters of length $n$ for a given consecutive pattern $\sigma$, with the additional value $s_{1,0} = 1$ (the purpose of this will soon be evident), and define $S_\sigma(t,x)$ to be the exponential generating function
\[S_\sigma(t,x) = \sum_{n,k} s_{n,k} t^k \frac{x^n}{n!}.\]
The following theorem, which can be proved via an inclusion-exclusion argument, relates $k$-clusters to consecutive pattern avoiding permutations.
\begin{thm}[\cite{GJ83}]
The e.g.f. for permutations avoiding the consecutive pattern $\sigma$ is given by
\[C_\sigma(x) = \frac{1}{1-S_\sigma(-1,x)}.\]
\end{thm}

\subsection{Class 4.V}\label{ssec:4classV}

In this section we will focus on c-Wilf class 4.V, and in particular on the enumeration of permutations avoiding the consecutive pattern 1423. The following arguments are discussed by \cite{EN12} in Section 5.1 and also by \cite{DK13} in Section 4.2.4.

First note that, as illustrated by the examples above, in a $k$-cluster with respect to 1423, two successive occurrences of 1423 can overlap by one or two elements. We can use this fact to build a recurrence for the sequence $s_{n,k}$, by conditioning on how many successive occurrences of 1423 overlap by two before the first successive pair which overlap by one. Let $\pi = (\pi_1,\pi_2,\ldots,\pi_n)$ be a $k$-cluster, with the first $\ell\leq k$ occurrences of 1423 overlapping by two, and then (if $\ell<k$) with the $\ell^\text{th}$ and $(\ell+1)^\text{th}$ occurrences overlapping by one.

Observe next that $\pi_1,\pi_3,\ldots,\pi_{2\ell+1}$ and $\pi_{2\ell+2}$ are fixed, and must take the values $1,2,\ldots,\ell+1$ and $\ell+2$ respectively.  The values $\pi_2,\pi_4,\ldots,\pi_{2\ell}$ can then take any values from the remaining $\ell+3,\ldots,n$ (and must then be arranged in decreasing order). Once these have been chosen, the remainder of the permutation $\pi_{2\ell+2},\pi_{2\ell+3},\ldots,n$ is (after standardisation\footnote{That is, after shifting all elements down so that they become a \emph{bona fide} permutation on $1,2,\ldots,n-2\ell-1$ while maintaining the same relative order.}) a $(k-\ell)$-cluster of length $n-2\ell-1$.

We thus obtain the recurrence
\begin{equation}\label{eqn:c1423_recurrence}
s_{n,k} = \sum_{4 \leq 2\ell+2 \leq n} \binom{n-\ell-2}{\ell} s_{n-2\ell-1,k-\ell} \quad \text{for } n\geq4,
\end{equation}
where we impose the ``initial conditions'' $s_{n,k}=0$ for $n\leq 3$, except for $s_{1,0} = 1$.

If we let $t_n = \sum_k s_{n,k} (-1)^k$, then define the ordinary generating function $T(x) = 1+\sum_n t_n x^n$.
(Note that $T$ can be viewed as the o.g.f. counterpart to the e.g.f. $S_{1423}(-1,x)$ defined earlier.) Multiplying~\eqref{eqn:c1423_recurrence} by $(-1)^k x^n$ and summing then gives the following functional equation, stated as equation (18) by \cite{EN12}.
\begin{equation}\label{eqn:c1423_fe}
T(x) = 1 + \frac{x}{1+x}T\left(\frac{x}{1+x^2}\right).
\end{equation}

Our first new result is to demonstrate that~\eqref{eqn:c1423_fe} can be solved by straightforward iteration. To lighten notation, define
\begin{align*}
A(x) &:= \frac{x}{1+x} \\
B(x) &:= \frac{x}{1+x^2} \\
B_j(x) &:= B(B_{j-1}(x)) \text{ with } B_0(x) := x \\
A_j(x) &:= A(B_j(x)).
\end{align*}
Then~\eqref{eqn:c1423_fe} can be rewritten as $T(x) = 1+ A(x) T(B(x))$.

\begin{thm}\label{thm:Tx_iteration}
The functional equation~\eqref{eqn:c1423_fe} has the solution
\begin{equation}\label{eqn:Tx_iteration}
T(x) = 1 + \sum_{n=0}^\infty \prod_{j=0}^n A_j(x) = 1 + \sum_{n=0}^\infty \prod_{j=0}^n \frac{B_j(x)}{1+B_j(x)}.
\end{equation}
\end{thm}

\begin{proof}
The result follows by repeatedly iterating $x\mapsto B(x)$. After $N$ iterations, one has the relation
\begin{equation}\label{eqn:Tx_iteration_N_times}
T(x) = 1 + \sum_{n=0}^{N-1}\prod_{j=0}^n A_j(x) + \left(\prod_{j=0}^N A_j(x) \right)T(B_{N+1}(x)).
\end{equation}
Now $A_j(x) = x-x^2+O(x^3)$ while $T(B_j(x)) = 1+x+O(x^3)$. So as $N\to\infty$ in~\eqref{eqn:Tx_iteration_N_times}, the coefficient of $T(B_{N+1}(x))$ vanishes (in the sense of formal power series), leaving only~\eqref{eqn:Tx_iteration}.
\end{proof}

It is often the case that generating functions with solutions like~\eqref{eqn:Tx_iteration} are very difficult to analyse. Fortunately, our solution to $T(x)$ is simple enough to allow us to prove some further properties.
In particular, the following verifies Conjecture 5.1 by \cite{EN12}.

\begin{thm}\label{thm:Tx_nonDfinite}
The generating function $T(x)$, satisfying~\eqref{eqn:c1423_fe} and solved by Theorem~\ref{thm:Tx_iteration}, has an infinite number of poles in the complex plane, and is thus not differentiably finite.
\end{thm}

The behaviour of $A(x)$ and $B(x)$ will be important, so we plot them in Figures~\ref{fig:Ax} and~\ref{fig:Bx}. In particular, over $x\in\mathbb{R}$,
\begin{itemize}
\item $A(x)$ has a simple pole at $x=-1$, and is monotonically increasing from $(-1,-\infty)$ to $(0,0)$,
\item $B(x)$ has a local (and global) minimum at $(-1,-1/2)$, 
\item $B(x)$ monotonically decreases from $(-\infty,0)$ to $(-1,-1/2)$, and then monotonically increases to $(0,0)$.
\end{itemize}
\begin{figure}
\centering
\begin{subfigure}{0.48\textwidth}
\includegraphics[width=\textwidth]{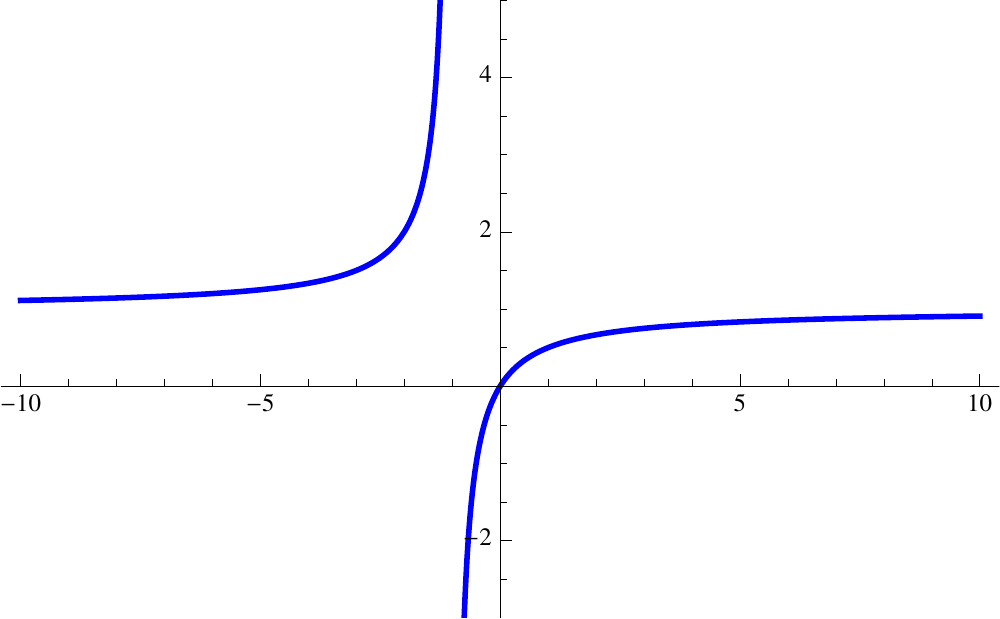}
\caption{}
\label{fig:Ax}
\end{subfigure}
\hspace{1em}
\begin{subfigure}{0.48\textwidth}
\includegraphics[width=\textwidth]{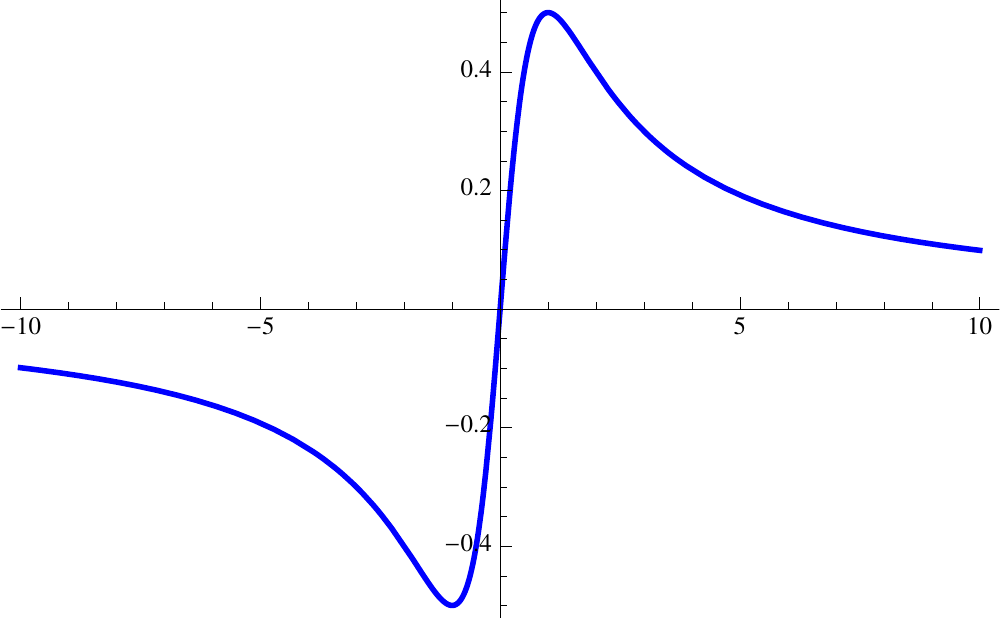}
\caption{}
\label{fig:Bx}
\end{subfigure}
\caption{Plots of $A(x)$ and $B(x)$ respectively.}
\end{figure}

Next, note that if $B(x) = t$ then
\begin{equation}\label{eqn:Cpm}
x = \frac{1 \pm \sqrt{1-4t^2}}{2t} =: C_\pm(t).
\end{equation}
We will be working in the complex plane, and take the square root in $C_\pm(t)$ to be the principal value.

For the time being we will operate under the assumption that the following proposition is true; we will return to it later.

\begin{prop}\label{prop:Bx_hasroot}
For every $j\geq0$, the equation
\[B_j(x) = -1\]
has at least one root in $x\in\mathbb C$. Moreover, this root is not a root of $B_{j'}(x)=-1$ for any $j' \neq j$.
\end{prop}

Now for given $J$, the terms in~\eqref{thm:Tx_iteration} which contain $A_J(x)$ as a factor can be written as
\begin{equation}\label{eqn:sum_terms_BJ}
T_J(x) := \left(\prod_{j=0}^{J-1}A_j(x)\right) \cdot A_J(x) \cdot \left(1 + \sum_{n=J+1}^\infty \prod_{j=J+1}^n A_j(x)\right).
\end{equation}
Let us write this as $T_J(x) = U_J(x) \cdot A_J(x) \cdot V_J(x)$.

Let $x_J$ be one of the roots of $B_J(x) = -1$. Observe that
\begin{align*}
B_{J+1}(x_J) &= B(B_J(x_J)) = B(-1) = -1/2 \\
B_{J+2}(x_J) &= B(B_{J+1}(x_J) = B(-1/2) = -2/5 \\
B_{J+3}(x_J) &= B(B_{J+2}(x_J) = B(-2/5) = -10/29
\end{align*}
and so on. It is clear (see Figure~\ref{fig:Bx}) that for $J\geq0$, the sequence $(B_{J+j}(x_J))_{j\geq1}$ is negative and increasing, with all values lying between $-1/2$ and 0. Then
\begin{align*}
A_{J+1}(x_J) &= A(-1/2) = -1 \\
A_{J+2}(x_J) &= A(-2/5) = -2/3 \\
A_{J+3}(x_J) &= A(-10/29) = -10/19
\end{align*}
and so on. It is then clear (see Figure~\ref{fig:Ax}) that for $J\geq0$, the sequence $(A_{J+j}(x_J))_{j\geq1}$ is negative and increasing, with all values lying between $-1$ and 0.

Consider then $V_J(x)$, evaluated at $x=x_J$. By the above arguments, we have
\begin{itemize}
\item $V_J(x_J)$ is independent both of $J$ and of the particular root $x_J$,
\item the $n^\text{th}$ term in the infinite sum in $V_J(x_J)$, for $n=J+1, J+2, \ldots$ is a product of $n-J$ increasing negative numbers, the first of which is $-1$.
\end{itemize}
Thus $V_J(x_J)$ converges to a number $v$, independent of $J$ and $x_J$. To 12 decimal places, this number is
\[v = 0.427119583148\ldots\]

This takes care of $V_J$, so we now turn to $U_J$. Since $U_J(x)$ is the finite product
\[\frac{B_0(x)}{1+B_0(x)} \cdot \frac{B_1(x)}{1+B_1(x)} \cdots \frac{B_{J-1}(x)}{1+B_{J-1}(x)},\]
we are concerned with two possibilities:
\begin{itemize}
\item Are any of $B_0(x_J)$, \ldots, $B_{J-1}(x_J)$ equal to 0?
\item Are any of $B_0(x_J)$, \ldots, $B_{J-1}(x_J)$ equal to $-1$?
\end{itemize}
It is straightforward to see that neither of these is possible. In the first case, if $B_{J-j}(x_J) = 0$ for some $j \geq 1$, then $B_{J-j+1}(x_J) = B(B_{J-j}(x_J)) = B(0) = 0$, and so on, eventually leading to $B_J(x_J) = 0$, which contradicts the fact that $B_J(x_J) = -1$.

In the second case, if $B_{J-j}(x_J) = -1$ for some $j\geq 1$, then $B_{J-j+1}(x_J) = B(B_{J-j}(x_J)) = B(-1) = -1/2$ (see the arguments above for $V_J$), and so on, eventually leading to $B_J(x_J) > -1$, again a contradiction.

Thus, it follows that $U_J(x_J)$ is some finite non-zero complex number. Unlike $V_J(x_J)$, $U_J(x_J)$ is dependent on $J$ and $x_J$.

\begin{proof}[of Theorem~\ref{thm:Tx_nonDfinite} (given Proposition~\ref{prop:Bx_hasroot})]
The above arguments establish that $T_J(x)$ has a pole at $x=x_J$. It remains to be seen that this pole is not cancelled by some other term in $T(x)$.

Define 
\[R_J(x) := T(x) - T_J(x) = 1 + \sum_{n=0}^{J-1} \prod_{j=0}^n A_j(x).\]
The pole in $T_J(x)$ at $x=x_J$ can only be cancelled if $R_J(x)$ also has a pole at $x=x_J$. But $R_J(x)$ is a finite sum, so it can only have a pole if one of its summands has a pole. The summands are finite products of $A_k(x)$, so if $R_J(x)$ has a pole at $x=x_J$ then one of the $A_k(x)$ must have a pole, ie.~one of $B_0(x), \ldots, B_{J-1}(x)$ must be equal to $-1$ at $x=x_J$. But it was established above that this is not possible. 

Hence $x=x_J$ is not only a pole of $T_J(x)$, but of $T(x)$ itself. Thus $T(x)$ has an infinite number of poles (subject to the proof of Proposition~\ref{prop:Bx_hasroot}.)
\end{proof}

\begin{proof}[of Proposition~\ref{prop:Bx_hasroot}]
Clearly $x_0 = -1$. For $j\geq1$, this just depends on repeated compositions of $C_\pm(x)$. Since we only need a single root, we can just stick with $C_+(x)$. Then we can use
\begin{align*}
x_1 &= C_+(-1) = -0.5-0.866025 i \\
x_2 &= C_+(C_+(-1)) = -0.351597+1.49853 i \\
x_3 &= C_+(C_+(C_+(-1))) = -0.0966266-1.36268 i
\end{align*}
and so on. Note that the arguments from the proof of Theorem~\ref{thm:Tx_nonDfinite} demonstrate that no root in this sequence can be repeated for different $j$.
\end{proof}

More generally, the full set of poles (ie.~all possible $x_j$) are generated by taking all possible combinations of $C_+$ and $C_-$. In this way, one finds
\begin{align*}
x_1 &\in \{C_\pm(-1)\} = \{-0.5\pm0.866025 i\} \\
x_2 &\in \{C_\pm(C_\pm(-1))\} = \{-0.351597\pm1.49853 i, -0.148403\pm0.632502 i\} \\
x_3 &\in \{C_\pm(C_\pm(C_\pm(-1)))\} = \begin{multlined}[t] \{-0.0966266\pm1.36268 i, -0.281881\pm1.99093 i, \\ -0.0517763\pm0.730177 i, -0.069716\pm0.492406 i\}\end{multlined}
\end{align*}
and so on.

\begin{prop}\label{prop:xj_neg_real}
For $j\geq 0$, let $x_j$ be any root of $B_j(x) = -1$. Then $\Re(x_j)<0$.
\end{prop}
\begin{proof}
First note that if $x=a+bi$ with $a,b\in\mathbb{R}$ and $x\neq \pm i$, then
\[B(x) = \frac{a(1+a^2+b^2) + b(1-a^2-b^2)i}{(1+a^2-b^2)^2 + (2ab)^2}.\]
In particular, $\Re(B(x))$ has the same sign as $\Re(x)$.

Now let $x_j$ satisfy $B_j(x_j) = -1$. Of course $B_j(x_j) = B(B_{j-1}(x_j))$, so by the above observation, we then have $\Re(B_{j-1}(x_j)) < 0$. Then $B_{j-1}(x_j) = B(B_{j-2}(x_j))$, so $\Re(B_{j-2}(x_j)) < 0$. This can be iterated all the way down to $B_0(x_j) = x_j$, so we conclude that $\Re(x_j) <0$.
\end{proof}

\begin{prop}\label{prop:2toj_roots}
For $j\geq0$ the equation $B_j(x) = -1$ has $2^j$ distinct roots.
\end{prop}
\begin{proof}
This is true for $j=0$ and $j=1$. So take $j\geq1$, and let $x_j$ and $x'_j$ be two distinct roots of $B_j(x) = -1$. By the above arguments, all of
\[C_+(x_j), C_-(x_j), C_+(x'_j) \text{ and } C_-(x'_j)\]
are roots of $B_{j+1}(x)=-1$.

Then (for a contradiction) suppose $C_+(x_j) = C_+(x'_j)$. Applying $B$ to both sides (noting that, by Proposition~\ref{prop:xj_neg_real}, neither of these is equal to $\pm i$) gives $x_j=x'_j$, a contradiction. The same argument can be applied to show that $C_-(x_j) \neq C_-(x'_j)$ and $C_+(x_j) \neq C_-(x'_j)$ (with $C_-(x_j) \neq C_+(x'_j)$ being a natural implication of the latter).

It remains to be shown that $C_+(x_j) \neq C_-(x_j)$ (with $C_+(x'_j) \neq C_-(x'_j)$ following from this). The only possible solutions to this equation are $x_j = \pm \frac12$. By Proposition~\ref{prop:xj_neg_real} we cannot have $x_j = \frac12$, so suppose that $x_j = -\frac12$. But this too is impossible: if $x_j = -\frac12$, then we know $B_j(x_j)$ for $j=0,1,2,\ldots$ is an increasing sequence (recall Figure~\ref{fig:Bx}), contradicting the fact that $B_j(x_j) = -1$.

Then by induction, there are twice as many roots of $B_{j+1}(x)=-1$ than $B_j(x) = -1$, and since there is one root for $j=0$, the proof is complete.
\end{proof}

\subsection{Generalising class 4.V to length $m$}\label{ssec:generalising_V}

For $m\geq4$ let $\tau_m$ be the consecutive pattern
\[1m23\ldots(m-2)(m-1).\]
So $\tau_4 = 1423$, $\tau_5=15234$, and so on. The arguments leading to~\eqref{eqn:c1423_recurrence} easily generalise, and it follows that the cluster counts $s_{n,k}$ satisfy the recurrence
\begin{equation}\label{eqn:4V_generalisation_recurrence}
s_{n,k} = \sum_{m \leq (m-2)\ell+2 \leq n} \binom{n-(m-3)\ell-2}{\ell} s_{n-(m-2)\ell-1,k-\ell} \quad \text{for } n\geq m,
\end{equation}
with the initial conditions $s_{n,k} = 0$ for $n<m$, except for $s_{1,0} = 1$.

The recurrence~\eqref{eqn:4V_generalisation_recurrence} can be translated into a functional equation in the same way as for class 4.V:
\begin{equation}\label{eqn:4V_generalisation_fe}
T^m(x) = 1 + \frac{x}{1+x}T^m\left(\frac{x}{1+x^{m-2}}\right).
\end{equation}
We will keep $A(x)$ defined as in the previous section, but now generalise $B(x), B_j(x)$ and $A_j(x)$:
\begin{align*}
B^m(x) &:= \frac{x}{1+x^{m-2}} \\
B^m_j(x) &:= B^m(B^m_{j-1}(x)) \text{ with } B^m_0(x) := x \\
A^m_j(x) &:= A(B^m_j(x))
\end{align*}
The following then follows {\em mutatis mutandis} from Theorem~\ref{thm:Tx_iteration}, so we omit the proof.
\begin{thm}\label{thm:4V_generalisation_iteration}
The functional equation~\eqref{eqn:4V_generalisation_fe} has the solution
\begin{equation}\label{eqn:4V_generalisation_iteration}
T^m(x) = 1 + \sum_{n=0}^\infty \prod_{j=0}^n A^m_j(x) = 1 + \sum_{n=0}^\infty \prod_{j=0}^n \frac{B^m_j(x)}{1+B^m_j(x)}.
\end{equation}
\end{thm}

Not only can we solve $T^m(x)$ for any $m\geq4$, we can also generalise Theorem~\ref{thm:Tx_nonDfinite}.

\begin{thm}\label{thm:Tmx_non_Dfinite}
The generating function $T^m(x)$, with the solution given by Theorem~\ref{thm:4V_generalisation_iteration}, has an infinite number of poles in the complex plane, and is thus not differentiably finite.
\end{thm}

Similar to the previous section, this result depends on the analogue of Proposition~\ref{prop:Bx_hasroot}, which we assume to be true for the time being.

\begin{prop}\label{prop:V_generalisation_B_hasroot}
For every $m\geq4$ and $j\geq0$, the equation
\[B^m_j(x)=-1\]
has at least one root in $x\in\mathbb C$. Moreover, this root is not a root of $B^m_{j'}(x)=-1$ for any $j' \neq j$.
\end{prop}

\begin{proof}[of Theorem~\ref{thm:Tmx_non_Dfinite} (given Proposition~\ref{prop:V_generalisation_B_hasroot})]
For even $m$, the proof of this theorem follows, \emph{mutatis mutandis}, the same ideas as for Theorem~\ref{thm:Tx_nonDfinite}. In particular, $B^m(x)$ is continuous and differentiable over all of $\mathbb R$, with a global minimum at $x=\rho_m \in [-1,-\frac34)$, and always satisfies $B^m(-1) = -\frac12$.

For odd $m$, things are a little different. First, generalise $T_J(x)$ as defined in~\eqref{eqn:sum_terms_BJ} to $T^m_J(x)$ in the obvious way, and let $x^m_J$ be one of the roots of $B^m_J(x) = -1$. Now $B^m_{J+1}(x)$ diverges as $x\to x^m_J$, but $B^m_{J+2}(x) \to 0$, as does every $B^m_{J+j}(x)$ for $j\geq2$. Meanwhile $A^m_{J+1}(x) \to 1$ as $x\to x^m_J$, and $A^m_{J+j}(x)\to0$ for all $j\geq2$. This implies that only the first term in the infinite sum in $V^m_J(x)$ remains non-zero as $x\to x^m_J$, and hence
\[V^m_J(x) \to 2 \quad \text{as } x\to x^m_J.\]

For $U^m_J(x)$, we have a similar argument to the previous section that all of the terms in the product are finite and non-zero: none of the $B^m_{J-j}(x)$ can be 0 at $x^m_J$ (or else so too would $B^m_J$), nor can they be $-1$ (or else $B^m_J$ would either be undefined or 0). So $U^m_J(x)$ is a finite, non-zero complex number at $x=x^m_J$.

The final part of the proof (that is, showing that the pole at $x=x^m_J$ in $T^m_J(x)$ is really a pole of $T^m(x)$) is exactly the same as that of Theorem~\ref{thm:Tx_nonDfinite}.
\end{proof}

\begin{proof}[of Proposition~\ref{prop:V_generalisation_B_hasroot}]
This follows the same ideas as for Proposition~\ref{prop:Bx_hasroot}. For Proposition~\ref{prop:Bx_hasroot} we used the explicit inverse function of $B(x)$; in general, however, the equation $B^m(x) = -1$ is a polynomial equation of degree $m-2$ (namely $x^{m-2}+x+1=0$). Nevertheless, we can still use the same process:
\begin{align*}
&x^m_1 \text{ can be any solution to } B^m(x) = -1 \\
&x^m_2 \text{ can be any solution to } B^m(x) = x^m_1 \\
&x^m_3 \text{ can be any solution to } B^m(x) = x^m_2
\end{align*}
and so on. By the same arguments used in the proof of Theorem~\ref{thm:Tmx_non_Dfinite}, no root in this sequence can be repeated.
\end{proof}

\subsection{A further generalisation}\label{ssec:further_generalisation}

The recurrence~\eqref{eqn:c1423_recurrence} and its generalisation~\eqref{eqn:4V_generalisation_recurrence} are in fact specialisations of an even more general set of recurrences, though in general we do not know how to solve them as we did in the previous sections.

This time, take $m\geq4$ and let $\tau$ be a pattern of the form
\[1m\tau_3\tau_4\ldots \tau_{m-1}\tau_m,\]
with the property that $\tau_m = \tau_{m-1}+1$. Set $c\equiv c(\tau) = m-\tau_m-1$.
\begin{thm}\label{thm:greater_generalisation}
The cluster counts $s_{n,k}$ for the consecutive pattern $\tau$ satisfy the recurrence
\begin{equation}\label{eqn:greater_generalisation_recurrence}
s_{n,k} = \sum_{m \leq (m-2)\ell+2 \leq n} \binom{n-(m-c-3)\ell-2}{(c+1)\ell}s_{n-(m-2)\ell-1,k-\ell} \quad\text{for } n\geq m,
\end{equation}
with the initial conditions $s_{n,k}=0$ for $n<m$, except for $s_{1,0} = 1$.
\end{thm}
Note that the consecutive patterns considered earlier were ``boundary cases'', with $c=0$.

\begin{proof}
Two successive occurrences of $\tau$ in a $k$-cluster can overlap by one or two elements, so we follow a similar strategy to the one outlined at the start of Section~\ref{ssec:4classV}. Let $\sigma$ be a $k$-cluster of length $n$ for $\tau$, and say the first $\ell$ occurrences of $\tau$ in $\sigma$ overlap by two elements, with the $\ell^\text{th}$ and $(\ell+1)^\text{th}$ occurrences overlapping by one. Let $\sigma_{a_1}, \sigma_{a_2},\ldots,\sigma_{a_s}$ be the elements of $\sigma$, up to and including $\sigma_{(m-2)\ell+2}$, which satisfy $\sigma_{a_i} \leq  \sigma_{(m-2)\ell+2}$. Similarly, let $\sigma_{b_1},\sigma_{b_2},\ldots,\sigma_{b_t}$ be the elements of $\sigma$, up to $\sigma_{(m-2)\ell+2}$, which satisfy $\sigma_{b_i} > \sigma_{(m-2)\ell+2}$.

By definition of $c$, we must have $t=(c+1)\ell$, and hence $s=(m-2)\ell+2-t = (m-c-3)\ell+2$. The set $\{\sigma_{a_1},\ldots\sigma_{a_s}\}$ is exactly the set $\{1,2,\ldots,(m-c-3)\ell+2\}$, with the order of the $\sigma_{a_i}$ depending on (and being fixed by) the particular $\tau$ in question. The $\sigma_{b_i}$ elements are then free to take any values between $(m-c-3)\ell+3$ and $n$, with their order again being fixed by $\tau$. The remainder of $\sigma$, i.e.~$\sigma_{(m-2)\ell+2},\ldots,\sigma_n$, is then (after standardisation) a $(k-\ell)$-cluster of length $n-(m-2)\ell-1$. The recurrence~\eqref{eqn:greater_generalisation_recurrence} follows.
\end{proof}

\section{Consecutive PAPs of length 5}\label{sec:length5}

In the case of classical PAPs, it is known (see \cite{K12})  that the 120 possible patterns of length five can be divided into 16 so-called Wilf classes, each of which has the same asymptotic behaviour. For consecutive, length five PAPs, there are twenty five (see \cite{E15}) equivalence classes, each of which has the same asymptotic behaviour of its coefficients. These are given below, ordered in terms of increasing coefficient size (not necessarily asymptotically, but lexicographically):
\begin{align*}
\text{5.I: } &12354 \sim 21345 \sim 45321 \sim 54312 \\
\text{5.II: } &12453 \sim 12543 \sim 31245 \sim 32145 \sim 34521 \sim 35421 \sim 54123 \sim 54213 \\
\text{5.III: } &21534 \sim 23154 \sim 43512 \sim 45132 \\
\text{5.IV: } &24153 \sim 25143 \sim 31524 \sim 32514 \sim 34152 \sim 35142 \sim 41523 \sim 45213  \\
\text{5.V: } &13452 \sim 13542 \sim 14352 \sim 14532 \sim 15342 \sim 15432 \sim 23451 \sim 23541 \\ &\sim 24351 \sim 24531 \sim 25341 \sim 25431 \sim 41235 \sim 41325 \sim 42135 \sim 42315 \\ &\quad\sim 43125 \sim 43215 \sim 51234 \sim 51324  \sim 52134 \sim 52314 \sim 53124 \sim 53214\\
\text{5.VI: } &12435 \sim 13245 \sim 53421 \sim 54231 \\
\text{5.VII: } &15234 \sim 23415 \sim 43251 \sim 51432  \\
\text{5.VIII: } &15423 \sim 32451 \sim 34215 \sim 51243 \\
\text{5.IX: } &21354 \sim 45312\\
\text{5.X: } &21453 \sim 31254 \sim 35412 \sim 45213 \\
\text{5.XI: } &13425 \sim 14235 \sim 52431 \sim 53241 \\
\text{5.XII: } &14523 \sim 32541 \sim 34125 \sim 52143 \\
\text{5.XIII: } &23514 \sim 25134 \sim 41532 \sim 43152 \\
\text{5.XIV: } &25413 \sim 31452 \sim 35214 \sim 41253 \\
\text{5.XV: } &15324 \sim 24315 \sim 42351 \sim 51342 \\
\text{5.XVI: } &12534 \sim 23145 \sim 43521 \sim 54132 \\
\text{5.XVII: } &21543 \sim 32154 \sim 34512 \sim 45123 \\
\text{5.XVIII: } &14325 \sim 52341 \\
\text{5.XIX: } &13524 \sim 24135 \sim 42531 \sim 53142 \\
\text{5.XX: } &25314 \sim 41352 \\ 
\text{5.XXI: } &24513 \sim 31542 \sim 35124 \sim 42153 \\
\text{5.XXII: } &13254 \sim 21435 \sim 45231 \sim 53412 \\
\text{5.XXIII: } &15243 \sim 32415 \sim 34251 \sim 51423 \\
\text{5.XXIV: } &14253 \sim 31425 \sim 35241 \sim 52413 \\
\text{5.XXV: } &12345 \sim 54321
\end{align*}

The solution of class 5.XXV follows from the first theorem of Elizalde and Noy. The exponential generating function is $1/w(x)$, where
\[w^{iv}+w^{'''}+w^{''}+w'+w = 0; \quad w(0)=1, \quad w'(0)=-1, \quad w^{''}(0) =0, \quad w^{'''}(0) =0,\]
 while the solution of classes 5.I, 5.II and 5.V follow from the second theorem. \cite{EN12} proved further results which solve classes 5.VI, 5.XVI and 5.XXII.  For class 5.XI Elizalde and Noy also prove results which lead to a $13^\text{th}$ order ODE for $1/w(z)$, which we improve to a $7^\text{th}$
order ODE in Section~\ref{sec:numerical_results}. The relevant ODEs are given in Table~\ref{table:length5_des} below.
We wish to point out that the solution to the ODE above for class 5.XXV is $$w(x)=\sum_{n \ge 0} \frac{(5n+1-x)\cdot x^{5n}}{(5n+1)!}.$$

For class 5.VII, the results of Section~\ref{ssec:generalising_V} give the solution to the cluster generating function, and demonstrate that it is non-D-finite. It follows that the e.g.f. for this class of c-PAPs is non-D-finite. Class 5.VIII is covered by the results of Section~\ref{ssec:further_generalisation}, and while we can in fact write down a functional equation for the cluster generating function, it is complicated and we do not know how to solve it. 

In the next section we find an explicit, algebraic solution to the cluster generating function for class 5.XI, and conjecture that a generalisation of this holds for certain classes of length $m\geq4$. We will also find a recurrence for the cluster numbers of class 5.XII, and a differential-functional equation for class 5.XXIII.

\subsection{Class 5.XI}

We will focus on the consecutive pattern 13425 from class 5.XI. This class, and the generalisations discussed in Section~\ref{ssec:XI_generalisation}, are covered by \cite{EN12} in Section 4.2. In that paper, Elizalde and Noy find explicit D-finite solutions to a family of cPAPs, of the form
\[134\ldots(s+1)2(s+2)(s+3)\ldots m,\]
where $m$ is the length of the pattern and $s+2 \leq m \leq 2s$. (Class 5.XI and the generalisations of the following section are the case $s=m-2$.) For clarity we will reproduce here some of the arguments behind \cite{EN12}'s Theorem 4.2, and Theorems~\ref{thm:13425_recurrence} and~\ref{thm:XI_generalisation_recurrence} are essentially rewordings of that theorem. Theorem~\ref{thm:13425_solution} and Conjectures~\ref{conj:XI_generalisation_solution} and~\ref{conj:XI_generalisation_algebraic} are new, however.

As usual, we count $k$-clusters of length $n$ by conditioning on the number $\ell$ of initial occurrences of 13425 which overlap by two elements. Let $\sigma$ be a $k$-cluster with $\ell\geq2$. There are three possible relative orderings for $\sigma_1,\ldots,\sigma_8$, namely
\begin{align}
&\sigma_1 < \sigma_4 < \sigma_2 < \sigma_3 < \sigma_7 < \sigma_5 < \sigma_6 < \sigma_8 \tag{5.XI-1}\label{eqn:5.XI-1}\\
&\sigma_1 < \sigma_4 < \sigma_2 < \sigma_7 < \sigma_3 < \sigma_5 < \sigma_6 < \sigma_8 \tag{5.XI-2}\label{eqn:5.XI-2}\\
&\sigma_1 < \sigma_4 < \sigma_7 < \sigma_2 < \sigma_3 < \sigma_5 < \sigma_6 < \sigma_8. \tag{5.XI-3}\label{eqn:5.XI-3}
\end{align}
Note that the relative ordering of all elements except $\sigma_7$ is the same throughout.

Similarly if $\ell\geq3$, the relative order of $\{\sigma_1,\ldots,\sigma_{11}\}\backslash\{\sigma_7,\sigma_{10}\}$ will be fixed. In the case~\eqref{eqn:5.XI-1}, there will be three possible positions for $\sigma_{10}$; in the case~\eqref{eqn:5.XI-2}, there will be four possible positions; and in case~\eqref{eqn:5.XI-3}, there will be five possible positions. This generalises as follows: if there are at least $\ell+1\geq3$ occurrences of 13425 at the start of $\sigma$ which overlap by two elements, consider the relative ordering of the first $3\ell+2$ elements, and in particular the position of $\sigma_{3\ell+1}$ in that ordering. It can occur anywhere between the $(\ell+1)^\text{th}$ and the $(3\ell-1)^\text{th}$ positions. Say $\sigma_{3\ell+1}$ is in the $p^\text{th}$ position in that ordering. Now also take into the account the positions of the next three terms of $\sigma$, i.e.~$\sigma_{3\ell+3}, \sigma_{3\ell+4}$ and $\sigma_{3\ell+5}$. The positions of $\sigma_{3\ell+3}$ and $\sigma_{3\ell+5}$ are fixed, but there are $3\ell+2-p$ possible positions for $\sigma_{3\ell+4}$.

This continues, where at each stage (i.e.~each additional occurrence of 13425 which overlaps its predecessor by two) we have between three and $2\ell+1$ possible choices for $\sigma_{3\ell+4}$. The reader may have observed the resemblance to the growth of ternary trees:
\begin{itemize}
\item Start with a tree with one node and three leaves, and label the leaves $1,2,3$ from left to right. All three leaves are currently \emph{active}.
\item Select leaf $i$ to convert into a node. Label $i$ vanishes, and all leaves with label $<i$ become \emph{inactive}. The new leaves are active, and get the labels $4,5,6$.
\item Repeat this process, each time selecting an active leaf to convert into a node, deactivating all leaves with smaller labels, and adding three new leaves at the new node.
\item After $n$ steps, there are $n$ nodes and $2n+1$ leaves. At least three must be active, but they may all be.
\end{itemize}
In this way every ternary tree can be constructed in a unique and traceable way. Since it is well-known that the number of ternary trees with $n$ nodes is $\frac{1}{2n+1}\binom{3n}{n}$, we have the following.
\begin{thm}[\cite{EN12}'s Theorem 4.2]
\label{thm:13425_recurrence}
The cluster counts $s_{n,k}$ for the consecutive pattern 13425 satisfy the recurrence
\begin{equation}\label{eqn:13425_recurrence}
s_{n,k} = \sum_{5\leq 3\ell+2 \leq n}\frac{1}{2\ell+1}\binom{3\ell}{\ell}s_{n-3\ell-1,k-\ell}\quad \text{for } n\geq5,
\end{equation}
with the initial conditions $s_{n,k}=0$ for $n<5$, except for $s_{1,0}=1$.
\end{thm}

Multiplying~\eqref{eqn:13425_recurrence} by $(-1)^k x^n$ and summing, we obtain (with the help of Mathematica) the equation
\begin{equation}\label{eqn:13425_fe}
T(x) = 1 + x + x\left[\frac{2}{\sqrt{3x^3}} \sinh\left(\frac13 \arcsinh\left(\frac{3\sqrt{3x^3}}{2}\right)\right) -1 \right]\left(T(x) - 1\right).
\end{equation}
\begin{thm}\label{thm:13425_solution}
The cluster generating function $T(x)$ for the consecutive pattern 13425 has the solution
\[T(x) = \frac{1+2x-F(x)}{1+x-F(x)}\]
where
\[F(x) = \frac{2}{\sqrt{3x}}\sinh\left(\frac13 \arcsinh\left(\frac{3\sqrt{3x^3}}{2}\right)\right).\]
It is an algebraic function, being the root of a cubic polynomial with coefficients in $\mathbb Z[x]$.
\end{thm}

\begin{proof}
The solution to $T(x)$ is easily obtained from~\eqref{eqn:13425_fe}. As for the algebraicity of $T(x)$, this follows from the algebraicity of $F(x)$, which in turn follows from an application of the standard formula $\arcsinh(y) = \log(y+\sqrt{y^2+1})$. We have found the polynomial of which $T(x)$ is a root with the aid of Mathematica; it is
\begin{multline*}
1 + 2x + 6x^2 + 12x^3 + 8x^4 - (3 + 5x + 15x^2 + 24x^3 + 12x^4)T \\ + (3 + 4x + 12x^2 + 15x^3 + 6x^4)T^2 - (1 + x + 3x^2 + 3x^3 + x^4)T^3. \qedhere
\end{multline*}
\end{proof}

Note that since $T(x)$ is algebraic, it is also D-finite, and hence so too is the e.g.f. for permutations avoiding the consecutive pattern 13425 (a fact already demonstrated by \cite{EN12}).

\subsection{Generalising class 5.XI to length $m$}\label{ssec:XI_generalisation}

It is natural to expect that the numbers $\frac{1}{2\ell+1}\binom{3\ell}{\ell}$ we observed in the last section to generalise. Indeed, it is well-known that $\frac{1}{(q-1)n+1}\binom{qn}{n}$ is the number of $q$-ary trees with $n$ nodes. It is straightforward to see that the length $m$ patterns which will give the same kind of recurrence as the trees are patterns of the form
\[\tau=134\ldots(m-1)2m\]
for $m\geq4$. For $m=4$ this is the pattern 1324, for $m=6$ it is 134526, and so on. The consecutive pattern 1324 falls into class 4.IV, whose cluster numbers are indeed known (see \cite{DK13,EN12})
to satisfy the recurrence
\[s_{n,k} = \sum_{4\leq2\ell+2 \leq n} \frac{1}{\ell+1}\binom{2\ell}{\ell} s_{n-2\ell-1,k-\ell}\]
with appropriate initial conditions. More generally, we have the following.
\begin{thm}[\cite{EN12}'s Theorem 4.2]
\label{thm:XI_generalisation_recurrence}
The cluster numbers $s_{n,k}$ for the consecutive pattern $\tau$, of length $m\geq4$, satisfy
\begin{equation}\label{eqn:XI_generalisation_recurrence}
s_{n,k} = \sum_{m \leq (m-2)\ell+2 \leq n} \frac{1}{(m-3)\ell+1}\binom{(m-2)\ell}{\ell}s_{n-(m-2)\ell-1,k-\ell} \quad \text{for } n\geq m,
\end{equation}
with initial conditions $s_{n,k}=0$ for $n<m$, except for $s_{1,0}=1$.
\end{thm}

With the aid of Mathematica, we then conjecture the form of the solution for arbitrary $m$.
\begin{conj}\label{conj:XI_generalisation_solution}
The cluster generating function $T(x)$ for the pattern $\tau$ of length $m\geq4$ is given by
\[T(x) = \frac{1+2x-G(x)}{1+x-G(x)}\]
where $G(x)$ is the generalised hypergeometric function
\begin{multline*}
G(x) = z\cdot_{(m-3)}\!F_{(m-4)}\left(\frac{1}{m-2},\frac{2}{m-2},\ldots,\frac{m-3}{m-2}; \frac{2}{m-3},\frac{3}{m-3},\ldots,\frac{m-4}{m-3},\frac{m-2}{m-3};\right. \\ \left.-\frac{(m-2)^{m-2}}{(m-3)^{m-3}}x^{m-2}\right),
\end{multline*}
using the notation
\[{}_pF_q(a_1,\ldots,a_p; b_1,\ldots,b_q; z) = \sum_{n=0}^\infty \frac{(a_1)_n \cdots (a_p)_n}{(b_1)_n \cdots (b_q)_n}\cdot \frac{z^n}{n!}.\]
Here, $(a)_n$ is the rising factorial
\[(a)_n = a(a+1)\cdots(a+n-1).\]
\end{conj}
We furthermore conjecture the following.
\begin{conj}\label{conj:XI_generalisation_algebraic}
For $m\geq4$, the solution to $T(x)$ given in Conjecture~\ref{conj:XI_generalisation_solution} is algebraic, being the root of a polynomial of degree $m-2$.
\end{conj}

We have numerically verified this conjecture for $m\leq 7$, finding the polynomials for $m=4$:
\[1 + 2 x + 4 x^2 +  4 x^3  -(2 + 3 x + 6 x^2 + 4 x^3) T + (1 + x + 2 x^2 + x^3) T^2;\]
for $m=6$:
\begin{multline*}
1 + 2 x + 8 x^2 + 24 x^3 + 32 x^4 + 16 x^5 -(4 + 7 x + 28 x^2 + 72 x^3 + 80 x^4 + 32 x^5) T \\ + (6 + 9 x + 36 x^2 + 78 x^3 + 72 x^4 + 24 x^5)T^2 - (4 + 5 x + 20 x^2 + 36 x^3 + 28 x^4 + 8 x^5)T^3 \\ + (1 + x + 4 x^2 + 6 x^3 + 4 x^4 + x^5) T^4;
\end{multline*}
and for $m=7$:
\begin{multline*}
1 + 2 x + 10 x^2 + 40 x^3 + 80 x^4 + 80 x^5 + 32 x^6 - (5 + 9 x + 45 x^2 + 160 x^3 + 280 x^4 + 240 x^5 + 80 x^6)T \\ + (10 + 16 x + 80 x^2 + 250 x^3 + 380 x^4 + 280 x^5 + 80 x^6)T^2 - (10 + 14 x + 70 x^2 + 190 x^3 + 250 x^4 + 160 x^5 + 40 x^6)T^3 \\ + (5 + 6 x + 30 x^2 + 70 x^3 + 80 x^4 + 45 x^5 + 10 x^6)T^4 - (1 + x + 5 x^2 + 10 x^3 + 10 x^4 + 5 x^5 + x^6)T^5.
\end{multline*}

\subsection{Class 5.XII}\label{ssec:class_XII}

We will focus on the consecutive pattern 14523 from class 5.XII. The recurrence for $s_{n,k}$ here is relatively straightforward. We first make the somewhat trivial observation that, if $\tau=14523$, then there are two elements of $\tau$ (namely $\tau_2$ and $\tau_3$) which are larger than $\tau_5$. Next, let $\sigma$ be a cluster with $\ell\geq2$ (using the usual definition of $\ell$). There are three possible relative orderings for the first 8 elements of $\sigma$:
\begin{align*}
&\sigma_1 < \sigma_4 < \sigma_7 < \sigma_8 < \sigma_5 < \sigma_2 < \sigma_3 < \sigma_6 \\
&\sigma_1 < \sigma_4 < \sigma_7 < \sigma_8 < \sigma_5 < \sigma_2 < \sigma_6 < \sigma_3 \\
&\sigma_1 < \sigma_4 < \sigma_7 < \sigma_8 < \sigma_5 < \sigma_6 < \sigma_2 < \sigma_3.
\end{align*}
The pertinent fact here is that there are always four elements which are bigger than $\sigma_8$. If $\ell\geq 3$ then there are 15 possible relative orderings for the first 11 elements of $\sigma$ (we will not list them all out), but one can check that there are always six elements bigger than $\sigma_{11}$. It is easy to see that this pattern continues: among the ordering of the first $3\ell+2$ elements of $\sigma$, there are always $2\ell$ elements which are bigger than $\sigma_{3\ell+2}$.

This matters for the following reason: say that the first $\ell+1$ occurrences of $\tau$ in $\sigma$ overlap by two elements, with the relative ordering of $\sigma_1,\ldots,\sigma_{3\ell+2}$ fixed, and consider the ways in which we can fit $\sigma_{3\ell+3},\sigma_{3\ell+4}$ and $\sigma_{3\ell+5}$. We  must have the ordering
\[\cdots < \sigma_{3\ell+1} < \sigma_{3\ell+4} < \sigma_{3\ell+5} < \sigma_{3\ell+2} < \cdots,\]
so our only choice is for the position of $\sigma_{3\ell+3}$. The only restriction is that it must be greater than $\sigma_{3\ell+2}$. We already know that there were previously $2\ell$ elements greater than $\sigma_{3\ell+2}$, so we have exactly $2\ell+1$ possible choices.

It hence follows that, with exactly the first $\ell$ occurrences of 14523 overlapping by two, there are
\[1\times3\times5\times\cdots\times(2\ell-1) = (2\ell-1)!!\]
choices for the relative ordering of $\sigma_1,\ldots,\sigma_{3\ell+2}$. Of those, the values of $\sigma_1,\sigma_4,\ldots,\sigma_{3\ell+1}$ and $\sigma_{3\ell+2}$ are entirely fixed, and take the values $1,2,\ldots,\ell+2$ respectively; we are free to choose the values of the remaining $2\ell$ elements. Thus, we have the following.
\begin{thm}\label{thm:XIII_recurrence}
The cluster numbers $s_{n,k}$ for the consecutive pattern 14523 satisfy the recurrence
\begin{equation}\label{eqn:XIII_recurrence}
s_{n,k} = \sum_{5 \leq 3\ell+2 \leq n} \binom{n-\ell-2}{2\ell}(2\ell-1)!! s_{n-3\ell-1,k-\ell} \quad \text{for } n\geq 5,
\end{equation}
with the initial conditions $s_{n,k} = 0$ for $n<5$, except for $s_{1,0} = 1$.
\end{thm}
Unfortunately, we do not know what kind of functional equation the corresponding generating function satisfies.

\subsection{Class 5.XXIII}\label{ssec:class_XXIII}

We will focus on the consecutive pattern 15243. This one is different to the others we have considered so far, because two successive occurrences of 15243 in a cluster can overlap by one or three elements, but not two. It is easy to see that if the first $\ell$ occurrences of 15243 overlap by three elements, then there is only one possible relative ordering of $\sigma_1,\ldots,\sigma_{2\ell+3}$. Moreover, the elements $\sigma_1,\sigma_3,\ldots,\sigma_{2\ell+3}$ are fixed, and take the values $1,2,\ldots,\ell+2$ respectively; we are free to choose the remaining $\ell+1$ terms. This leads to the following.
\begin{thm}\label{thm:XXIII_recurrence}
The cluster numbers $s_{n,k}$ for the consecutive pattern 15243 satisfy the recurrence
\begin{equation}\label{eqn:XXIII_recurrence}
s_{n,k} = \sum_{5\leq2\ell+3\leq n} \binom{n-\ell-2}{\ell+1} s_{n-3\ell-2,k-\ell} \quad \text{for }n\geq5,
\end{equation}
with initial conditions $s_{n,k}=0$ for $n<5$, except for $s_{1,0} = 1$.
\end{thm}
Multiplying~\eqref{eqn:XXIII_recurrence} by $(-1)^k x^n$ and summing, we arrive at the following differential-functional equation.
\begin{thm}\label{thm:XXIII_fe}
The cluster generating function $T(x)$ for the consecutive pattern 15243 satisfies the equation
\begin{equation}\label{eqn:XXIII_fe}
x^3T'(x) = 1 + x -T\left(\frac{x}{1+x^2}\right)
\end{equation}
with initial condition $T(0) = 1$.
\end{thm}
Unfortunately, we do not know how to solve~\eqref{eqn:XXIII_fe}.

\section{Generating series}\label{sec:gen_series}
\cite{BNZ} perform enumeration by a simple recursive technique using dynamic programming resulting
in polynomial time algorithms. The algorithm used to generate these has typically been the cluster method (see \cite{GJ79, EN12, N11}). This is a very different algorithm to ours, that runs in good time, and can, with some computational constraints, be used to prove that two series are in the same Wilf class. The faster algorithm we present below is simpler and easier to analyse, but is not as amenable to definitively proving that two series are in the same Wilf class. However, it can be used to generate large numbers of terms which provides strong evidence for Wilf equivalence.

Suppose one wants to enumerate permutations of length $n$ not including a consecutive pattern of length $L$.
A simple recursive algorithm would build up the permutation one step at a time, at each point
considering all the remaining integers, and rejecting it if, when added to the preceding $n-1$
elements, it produces the undesired pattern.

This simple algorithm will take time proportional to the number of such permutations. A more
efficient algorithm can be created by discarding unnecessary states and using dynamic programming.
In particular, the only thing that matters about elements prior to the last $L-1$ elements is that
those integers are already taken. Even this is more information than is needed; all one cares
about is how many integers are left in the gaps between the last $L-1$ integers. Similarly
the actual values of the last $L-1$ elements are not needed, only their relative order.

The number of such states is $O(L! n^{L-1})$, which is polynomial in $n$, making such
sequences easy to enumerate.

Furthermore a trivial alteration to the algorithm is to see if there are any possible
future integers that, added to the last $L-1$ elements, would produce the undesired pattern.
If so, the whole state is rejected. This causes the enumeration of permutations avoiding
$L$ elements, all consecutive except (possibly) the last.

In this way we have generated  series avoiding consecutive permutations of length-4 up to eighty terms, in some cases, and over 100 terms in others, and of length-5 up to seventy terms.

\section{Numerical results}\label{sec:numerical_results}

\subsection{Searching for differential equations}\label{ssec:finding_des}

The D-finiteness of the generating functions for classes 4.I, 4.IV, 4.VI and 4.VII was discussed in Section~\ref{sec:length4}, and the non-D-finiteness of class 4.V was proved in Section~\ref{ssec:4classV}. For the two remaining classes, namely 4.II and 4.III, we have not found a D-finite e.g.f., nor for its reciprocal. It is likely that the corresponding reciprocal of the e.g.f.~is not D-finite in those cases. As far as can be tested with the available series, it is not D-algebraic either.

For consecutive patterns of length 5, we have used the series data to rederive known differential equations satisfied by eight of the twenty-five generating functions. In each case the exponential generating function is given by $ \frac{1}{y(x)},$ where $y(x)$ is D-finite. The eight classes, and their corresponding differential equations, are given in Table~\ref{table:length5_des}. Note that the differential equation for class 5.XI is
\begin{multline}\label{eqn:DE_5.XI_long}
(3600 x+7920 x^2-1620 x^3-3240 x^4-18549 x^6)y(x) \\
+ (4480+10800 x+21840 x^2-1020 x^3-4224 x^4-55647 x^6) y'(x) \\
+ (13440+13520 x+18000 x^2+3540 x^3+6768 x^4-12366 x^5-55647 x^6) y''(x) \\
+ (13440 + 11760 x + 2480 x^2 + 540 x^3 + 12768 x^4 - 37098 x^5 - 21297 x^6) y'''(x) \\
+ (4480 + 11760 x + 6960 x^2 - 7140 x^3 + 816 x^4 - 37098 x^5 - 26793 x^6) y^{(4)}(x) \\
+ (4480 + 2720 x - 960 x^2 + 720 x^3 + 4056 x^4 - 12366 x^5 - 8244 x^6) y^{(5)}(x) \\
+ (2720 x + 320 x^2 - 3120 x^3 - 480 x^4 - 12366 x^5 - 2748 x^6) y^{(6)}(x) \\
+ (320 x^2 - 480 x^4 - 2748 x^6) y^{(7)}(x) = 0
 \end{multline}

\begin{table}
\setlength\abovedisplayskip{-10pt}
\setlength\belowdisplayskip{-10pt}
\setlength\abovedisplayshortskip{-10pt}
\setlength\belowdisplayshortskip{-10pt}
\begin{center}
\begin{tabular}{| c | m{20em} | m{15em} |}
\hline
%\begin{center}Class\end{center} & \begin{center}Differential equation\end{center} & \begin{center}Initial conditions\end{center} \\
Class & \begin{tightcenter}Differential equation\end{tightcenter}& \begin{tightcenter}Initial conditions\end{tightcenter} \\
\hline
5.I &\[ xy'(x) +y^{(4)} (x)  = 0\] & \begin{multline*} y(0) =1,\,y'(0) =-1, \\ y''(0) =0, \,y'''(0) =0 \end{multline*}\\
\hline
5.II & \[ {x}^{2}y'(x) +2y''' (x) = 0 \] & \[y(0) = 1,\, y'(0) = -1, \, y''(0) = 0\] \\
\hline
5.V & \[ x^3y'(x) + 6y''(x) = 0\] & \[ y(0) = 1, \, y'(0) = -1\]\\
\hline
5.VI& \[ y(x) + y'(x) + y^{(4)}(x) = 0\] & \begin{multline*} y(0) = 1,\, y'(0) = -1, \\ y''(0) = 0, \,y'''(0) = 0 \end{multline*}\\
\hline
5.XI & \begin{center}\eqref{eqn:DE_5.XI_long}\end{center} & \begin{multline*} y(0) = 1,\, y'(0) = -1, \\ y''(0) = 0,\, y'''(0) = 0, \\ y^{(4)}(0) = 0,\, y^{(5)}(0) = 1, \\ y^{(6)}(0) = 0 \end{multline*}\\
\hline
5.XVI & \[xy'(x) + xy''(x) +y^{(4)}(x) = 0\] & \begin{multline*} y(0) = 1,\, y'(0) = -1, \\ y''(0) = 0,\, y'''(0) = 0\end{multline*} \\
\hline
5.XXII & \[xy'(x) + y''(x) +y^{(4)}(x) = 0\] & \begin{multline*} y(0) = 1,\, y'(0) = -1, \\ y''(0) = 0,\, y'''(0) = 0\end{multline*}\\
\hline
5.XXV & \[y(x) + y'(x) +y''(x) + y'''(x) + y^{(4)}(x) = 0\] & \begin{multline*} y(0) = 1,\, y'(0) = -1, \\ y''(0) = 0,\, y'''(0) = 0\end{multline*} \\
\hline
 \end{tabular}
 \caption{The exponential generating function $C_{class}(x)=\frac{1}{y(x)}$, where $y(x)$ satisfies the given ODE for eight of the twenty-five classes of length 5 consecutive PAPs.}
 \label{table:length5_des}
 \end{center}
 \end{table}

We have not found a  D-finite  e.g.f., or its reciprocal, for the remaining seventeen classes. It seems likely that the corresponding reciprocal of the e.g.f. is not D-finite in those cases. As far as can be tested with the available series, it is not D-algebraic either. Thus it appears that \cite{EN03, EN12} have proved results for all the D-finite cases of c-PAPs of length up to 5. Our only new numerical result is to reduce the order of the differential equation for class 5.XI from 13 to 9. We subsequently were able to reduce this from 9 to 7. Then in Section~\ref{ssec:XI_generalisation}
we give an algebraic solution to the o.g.f. for this case.

\subsection{Estimating the growth rates}\label{ssec:growth_rates}

For each class $\Sigma$ of c-PAPs, the coefficients of the ordinary generating function behave asymptotically as 
\[c_n^\Sigma \sim C_\Sigma \cdot n! \cdot \kappa_\Sigma^n.\]
The estimate of the growth constant $\kappa_\Sigma$ and amplitudes $C_\Sigma$, accurate to all quoted digits, for all length 4 classes is given in Table~\ref{tab:l4growth}, and for all length 5 classes in Table~\ref{tab:l5growth}.
\begin{table}
\begin{center}
\begin{tabular}{|c |c|c|}
\hline
Class $\Sigma$ & $\kappa_\Sigma$&$C_\Sigma$ \\
\hline
 Class I & 0.9630055289154941756& 1.076344539715227\\
Class II & 0.9577180134976572362&1.137593123292952\\
Class III & 0.9561742431150784277 &1.146540529900785\\
Class IV & 0.9558503134742499890 &1.100226245067883\\
Class V & 0.9548260509498783340 &1.104489004860327\\
Class VI& 0.9546118344740519438 &1.103720832998758\\
Class VII& 0.9528914233250531974 &1.114556873900595\\
\hline
\end{tabular}
\end{center}
    \caption{Growth constants $\kappa_\Sigma$ and amplitudes $C_\Sigma$ for the seven classes of length 4 consecutive PAPs.}
    \label{tab:l4growth}
 \end{table}

%For each class of PAPs, the coefficients of the ordinary generating function behave asymptotically as $$ c_n^{class} \sim c(class)\cdot n! \cdot \kappa(class)^n,$$ where {\em class} denotes one of the twenty-five classes listed above. The estimate of the growth constant, accurate to all quoted digits, for all classes is given in Table \ref{tab:l5growth} below.

\begin{table}
\begin{center}
\begin{tabular}{|l |c|c|}
\hline
Class $\Sigma$ & $\kappa_\Sigma$& $C_\Sigma$ \\
\hline
 Class I & 0.9913880716699268181&1.0359338947290985\\
Class II & 0.9914185408600983479&1.0356740409503498\\
Class III & 0.9914215726255505158 & 1.0356482525747201 \\
Class IV & 0.9914637023566386736&1.0352912840051055 \\
Class V & 0.9914787346349870644 &1.0351640090771068\\
Class VI& 0.9914031046134865367 &1.0358339838201155\\
Class VII& 0.9914152799149738845 &1.0357301469691008\\
Class VIII& 0.9914455405535310693&1.0354727912203914\\
Class IX& 0.9914486888810151958&1.0354456527948576 \\
Class X& 0.9914905951981662739 &1.0350913714694614\\
Class XI& 0.9914573454495660358&1.0354283564589345 \\
Class XII& 0.9914991759877895239&1.0350742999782649\\
Class XIII& 0.9915021807789432127& 1.0350488441916296\\
Class XIV& 0.9915430268589110657&1.0347070291236631 \\
Class XV& 0.9914961218699849309&1.0351275914657668\\
Class XVI& 0.9914962152197285242& 1.0351265076491731\\
Class XVII& 0.9915702712612490911 &1.0345028067355504\\
Class XVIII& 0.9915374435675450185& 1.0348337036858431\\
Class XIX& 0.9915491202315941687& 1.0347354953793692\\
Class XX& 0.9915807073163505786& 1.0344726469008412\\
Class XXI& 0.9916208625283576837& 1.0341385793668625\\
Class XXII& 0.9916009188510841597&1.0345693190404323 \\
Class XXIII& 0.9916298962721992117& 1.0343256965190087 \\
Class XXIV& 0.9918325187738895504&1.0330524632572689 \\
Class XXV& 0.9928637443921790385 &1.0280679375675015\\
\hline
\end{tabular}
\end{center}
    \caption{Growth constants $\kappa_\Sigma$ and amplitudes $C_\Sigma$ for the twenty-five classes of length 5 consecutive PAPs.}
    \label{tab:l5growth}
 \end{table}

\section{Conclusion}
Of the seven length-4 consecutive-Wilf classes we give the ODEs and comment on the solution for four of these, and give a functional equation, solution and proof of non-D-finiteness for a fifth class.  For the twenty-five length-5 consecutive-Wilf classes we give the ODEs  for eight of these, and give a functional equation, solution and proof of non-D-finiteness for class 5.VII. For class 5.VIII we give a recurrence for cluster counts, and comment that this yields an ugly and unsolved recurrence for the cluster generating function. For class 5.XI we reduce the known ODE for the e.g.f from order 13 to order 7, and then give an algebraic solution to the o.g.f.
For class 5.XII we give a recurrence for cluster counts, but cannot solve this. For class 5.XIII we give a recurrence for cluster counts and show that this yields a differential-functional equation which we can't solve. 

Thus we have improved the state of our knowledge for one length-4 class and for five length-5 classes.
 
We have given a  fast, polynomial-time algorithm to generate the coefficients, and used this to (a) calculate the asymptotics for all classes of length 4 and length 5 to significantly greater precision than has previously been reported, and (b) use these extended series to search, unsuccessfully, for D-finite solutions for the unsolved classes, leading us to conjecture that the solutions are not D-finite. We have also searched, unsuccessfully, for differentially algebraic solutions. 

\acknowledgements
We would like to thank Sergi Elizalde for comments on an early version of this paper, which led to considerable improvement.

\nocite{*}
\bibliographystyle{abbrvnat}
% use the following instead if you encounter problems 
%\bibliographystyle{alpha}
\bibliography{cPAPs_FINAL}

\begin{thebibliography}{19}
\providecommand{\natexlab}[1]{#1}
\providecommand{\url}[1]{\texttt{#1}}
\expandafter\ifx\csname urlstyle\endcsname\relax
  \providecommand{\doi}[1]{doi: #1}\else
  \providecommand{\doi}{doi: \begingroup \urlstyle{rm}\Url}\fi

\bibitem[Aldred et~al.(2010)Aldred, Atkinson, and McCaughan]{AAM10}
R.~Aldred, M.~Atkinson, and D.~McCaughan.
\newblock Avoiding consecutive patterns in permutations.
\newblock \emph{Adv. Appl. Math.}, 45:\penalty0 449--461, 2010.

\bibitem[Arratia(1999)]{RA99}
R.~Arratia.
\newblock On the {Stanley-Wilf} conjecture for the number of permutations
  avoiding a given pattern.
\newblock \emph{Elec. J. Combin.}, 6:\penalty0 \#N1, 1999.

\bibitem[Baxter et~al.(2011{\natexlab{a}})Baxter, Nakamura, and
  Zeilberger]{Zeqns}
A.~Baxter, B.~Nakamura, and D.~Zeilberger, 2011{\natexlab{a}}.
\newblock Differential equations accompanying \cite{BNZ}.
  \url{www.math.rutgers.edu/~zeilberg/tokhniot/sergi/oET4_60},
  \url{www.math.rutgers.edu/~zeilberg/tokhniot/sergi/oET5_40}.

\bibitem[Baxter et~al.(2011{\natexlab{b}})Baxter, Nakamura, and
  Zeilberger]{Zseries}
A.~Baxter, B.~Nakamura, and D.~Zeilberger, 2011{\natexlab{b}}.
\newblock Series data accompanying \cite{BNZ}.
  \url{www.math.rutgers.edu/~zeilberg/tokhniot/sergi/oE4_60},
  \url{www.math.rutgers.edu/~zeilberg/tokhniot/sergi/oE5_40}.

\bibitem[Baxter et~al.(2013)Baxter, Nakamura, and Zeilberger]{BNZ}
A.~Baxter, B.~Nakamura, and D.~Zeilberger.
\newblock {Automatic Generation of Theorems and Proofs on Enumerating
  Consecutive-Wilf Classes}.
\newblock In I.~Kotsireas and E.~Zima, editors, \emph{Advances in
  Combinatorics}, pages 121--138. Springer Berlin Heidelberg, 2013.

\bibitem[Dotsenko and Khoroshkin(2013)]{DK13}
V.~Dotsenko and A.~Khoroshkin.
\newblock Shuffle algebras, homology, and consecutive pattern avoidance.
\newblock \emph{Algebra \& Number Theory}, 7:\penalty0 673--700, 2013.

\bibitem[Ehrenborg et~al.(2011)Ehrenborg, Kitaev, and Perry]{EKP}
R.~Ehrenborg, S.~Kitaev, and P.~Perry.
\newblock A spectral approach to consecutive pattern-avoiding permutations.
\newblock \emph{J. Combin.}, 2:\penalty0 305--353, 2011.

\bibitem[Elizalde(2006)]{E06}
S.~Elizalde.
\newblock Asymptotic enumeration of permutations avoiding generalized patterns.
\newblock \emph{Adv. Appl. Math.}, 36:\penalty0 138--155, 2006.

\bibitem[Elizalde(2012)]{E12}
S.~Elizalde.
\newblock The most and least avoided consecutive patterns.
\newblock \emph{Proc. Lond. Math. Soc.}, 106:\penalty0 957--979, 2012.

\bibitem[Elizalde(2015)]{E15}
S.~Elizalde.
\newblock A survey of consecutive patterns in permutations, 2015.
\newblock arXiv:1504.07265.

\bibitem[Elizalde and Noy(2003)]{EN03}
S.~Elizalde and M.~Noy.
\newblock Consecutive patterns in permutations.
\newblock \emph{Adv. Appl. Math.}, 30:\penalty0 110--125, 2003.

\bibitem[Elizalde and Noy(2012)]{EN12}
S.~Elizalde and M.~Noy.
\newblock Clusters, generating functions and asymptotics for consecutive
  patterns in permutations.
\newblock \emph{Adv. Appl. Math.}, 49:\penalty0 351--374, 2012.

\bibitem[Goulden and Jackson(1979)]{GJ79}
I.~P. Goulden and D.~M. Jackson.
\newblock An inversion theorem for cluster decompositions of sequences with
  distinguished subsequences.
\newblock \emph{J. Lond. Math. Soc.}, s2-20:\penalty0 567--576, 1979.

\bibitem[Goulden and Jackson(2004)]{GJ83}
I.~P. Goulden and D.~M. Jackson.
\newblock \emph{Combinatorial Enumeration}.
\newblock Dover Mineola, 2004.
\newblock Reprint of the 1983 original.

\bibitem[Guttmann(1989)]{G89}
A.~J. Guttmann.
\newblock In C.~Domb and J.~Lebowitz, editors, \emph{Phase Transitions and
  Critical Phenomena}. Academic Press, 1989.

\bibitem[Kitaev(2011)]{K12}
S.~Kitaev.
\newblock \emph{Patterns in Permutations and Words}.
\newblock Springer Berlin Heidelberg, 2011.

\bibitem[Marcus and Tardos(2004)]{MT04}
A.~Marcus and G.~Tardos.
\newblock Excluded permutation matrices and the {Stanley-Wilf} conjecture.
\newblock \emph{J. Combin. Theory Ser. A}, 107:\penalty0 153--160, 2004.

\bibitem[Nakamura(2011)]{N11}
B.~Nakamura.
\newblock Computational approaches to consecutive pattern avoidance in
  permutations.
\newblock \emph{Pure Math. Appl. (PU.M.A.)}, 22:\penalty0 253--268, 2011.

\bibitem[Perarnau(2013)]{P13}
G.~Perarnau.
\newblock A probabilistic approach to consecutive pattern avoiding in
  permutations.
\newblock \emph{J. Combin. Theory Ser. A}, 120:\penalty0 998--1011, 2013.

\end{thebibliography}
\label{sec:biblio}

\end{document}